\documentclass[12pt]{amsart}
\usepackage{amsthm, amsmath, amssymb, enumerate}
\usepackage{mathrsfs}
\usepackage{fullpage}
\usepackage{amscd}

\usepackage[colorlinks=true, allcolors=black]{hyperref}
\usepackage{MnSymbol}
\usepackage[all]{xy}
\usepackage{xr-hyper}
\usepackage{hyperref}
\usepackage{color}
\usepackage{cases}
\usepackage{xfrac}

\usepackage[colorinlistoftodos, textsize=tiny]{todonotes}
\def\listtodoname{List of Todos}
\def\listoftodos{\@starttoc{tdo}\listtodoname}

\theoremstyle{plain}

\newtheorem{thm}{Theorem}[section]
\newtheorem{theorem}[thm]{Theorem}
\newtheorem{cor}[thm]{Corollary}
\newtheorem{corollary}[thm]{Corollary}
\newtheorem{lem}[thm]{Lemma}
\newtheorem{lemma}[thm]{Lemma}

\newtheorem{proposition}[thm]{Proposition}
\newtheorem{prop}[thm]{Proposition}
\newtheorem{definition}[thm]{Definition}

\theoremstyle{remark} 
\newtheorem{remark}[thm]{Remark}
\newtheorem{example}[thm]{Example}

\numberwithin{equation}{section}

\newcommand{\ZZ}{\mathbb{Z}}

\newcommand{\RR}{\mathbb{R}}

\newcommand{\NT}{\mathrm{NT}}

\newcommand{\lra}{{\longrightarrow}}

\newcommand{\barint}{\mbox{$ave \int$}}  
\def\barint_#1{\mathchoice
	{\mathop{\vrule width 6pt
			height 3 pt depth -2.5pt
			\kern -8.8pt
			\intop}\nolimits_{#1}}%
	{\mathop{\vrule width 5pt height
			3 pt depth -2.6pt
			\kern -6.5pt
			\intop}\nolimits_{#1}}%
	{\mathop{\vrule width 5pt height
			3 pt depth -2.6pt
			\kern -6pt
			\intop}\nolimits_{#1}}%
	{\mathop{\vrule width 5pt height
			3 pt depth -2.6pt
			\kern -6pt \intop}\nolimits_{#1}}}

\newcommand{\BA}{{\mathbb {A}}}

\newcommand{\BC}{{\mathbb {C}}}

\newcommand{\BF}{{\mathbb {F}}}

\newcommand{\BN}{{\mathbb {N}}}

\newcommand{\BP}{{\mathbb {P}}}
\newcommand{\BQ}{{\mathbb {Q}}}
\newcommand{\BR}{{\mathbb {R}}}

\newcommand{\BV}{{\mathbb {V}}}

\newcommand{\BZ}{{\mathbb {Z}}}

\newcommand{\CA}{{\mathcal {A}}}

\newcommand{\CC}{{\mathcal {C}}}

\newcommand{\CE}{{\mathcal {E}}}
\newcommand{\CF}{{\mathcal {F}}}

\newcommand{\CL}{{\mathcal {L}}}

\newcommand{\CO}{{\mathcal {O}}}

\newcommand{\CS}{{\mathcal {S}}}

\newcommand{\Aut}{{\mathrm{Aut}}}

\renewcommand{\div}{{\mathrm{div}}}

\newcommand{\End}{{\mathrm{End}}}

\newcommand{\Gal}{{\mathrm{Gal}}}
\newcommand{\GL}{{\mathrm{GL}}}

\newcommand{\Id}{{\mathrm{Id}}}
\newcommand{\Jac}{{\mathrm{Jac}}}

\newcommand{\Pic}{\mathrm{Pic}}

	\newcommand{\SL}{{\mathrm{SL}}}

	\newcommand{\tor}{{\mathrm{tor}}}

	\newcommand{\vol}{{\mathrm{vol}}}
	
	\newcommand{\wt}{\widetilde}

	\newcommand{\norm}[1]{\|{#1}\|}

	\newcommand{\incl}{\hookrightarrow}

	\newcommand{\iso}{\overset \sim \lra}

	\newcommand{\sM}{{\mathscr {M}}}
	\newcommand{\sC}{{\mathscr {C}}}
	\newcommand{\sJ}{{\mathscr {J}}}
	\newcommand{\sA}{{\mathscr {A}}}
	
	\newcommand{\sN}{{\mathscr {N}}}
	\newcommand{\sE}{{\mathscr {E}}}
	
	\newcommand{\sP}{{\mathscr {P}}}
		\newcommand{\sK}{{\mathscr {K}}}
	
	\newcommand{\betti}{\mathrm{betti}}

	\newcommand{\hFal}{h_{\mathrm{Fal}}}
	
\begin{document}
		\title{Bigness of Canonical Quadratic Points on Curves of Genus 4}
			\author{Jiahui Gao}
		\maketitle
\setcounter{tocdepth}{1}

	\begin{abstract}
		A central problem in arithmetic geometry is to construct non-torsion rational points on elliptic curves. We study a canonical quadratic point $\xi_C \in \Jac(C)$ attached to a smooth non-hyperelliptic curve of genus $4$ and use it to produce such points on elliptic curves arising from families of genus $4$ curves. We introduce a notion of bigness for sections of abelian schemes and establish a criterion in terms of modular variation of abelian quotients, using adelic line bundles and Betti maps. As applications, we prove that $\xi_C$ is big on the triple-involution locus and on certain CM families, obtaining in particular non-torsion rational points on the associated elliptic curves and Northcott-type finiteness results.
	\end{abstract}
	
\tableofcontents
	
	\section{Introduction}
		Let $K$ be a field, and let $C$ be a smooth non-hyperelliptic curve of genus 4 over $K$. The canonical linear system realizes $C$ as a smooth curve in \(\BP^3\), which is the complete intersection of a unique quadric $Q$ and a cubic. When the quadric $Q$ is smooth, it becomes isomorphic to \(\BP^1 \times \BP^1\) over an extension $K'/K$ of degree $1$ or $2$ . The two rulings then induce two morphisms
	$$
	\pi_1,\pi_2 : C_{K'} \longrightarrow \BP^1_{K'}
	$$
	of degree 3.
	
		Fix a point $\infty \in \BP^1(K')$. The divisor classes $\pi_1^*(\infty)$ and $\pi_2^*(\infty)$ define points of $\Pic^3(C_{K'})$, and their difference determines a point
	$$
	\xi_C := [\pi_1^*(\infty)]-[\pi_2^*(\infty)] \in \Jac(C)(K').
	$$
	This point is independent of the choice of $\infty$, and is well defined up to sign. We refer to $\xi_C$ as the \emph{canonical quadratic point} attached to $C$.
	
	Let $\sM$ denote the moduli space of smooth curves of degree $(3, 3)$ in $\BP^1\times \BP^1$ modulo $\Aut(\BP^1)\times \Aut(\BP^1)$ over $\BQ$. Then we have an \'{e}tale cover  $\sM\to \sM_{4, \BQ}$ whose image is exactly the open subset of non-hyperelliptic curves with smooth quadric $Q$'s. Let 
	$$\pi: \sC\to \sM$$
	be the universal curve, and let 
	$$\sJ\to \sM$$
	be the relative Jacobian. The canonical quadratic points over the fibers  define a section 
	$$\xi: \sM\to \sJ.$$

A theorem of Xue\cite{Xu15} shows that \(\xi\) is generically non-torsion over the moduli space of smooth genus \(4\) curves:

\begin{theorem}[Xue]\label{thm_xue}
Let $K$ be the function field of the generic point $\eta$ of $\sM$ and $C/K$ a curve representing $\eta$. Then the canonical quadratic point $\xi_\eta=\xi (\eta)$ of $\Jac(C)$ is non-torsion. Moreover, the group $\Jac(C)(K')$ is a free group generated by $\xi_\eta$. 
\end{theorem}

 The purpose of this paper is to extend the above non-degeneracy of $\xi$ to general algebraic points on an irreducible subvariety $\sM$. For this, we define a notion of bigness for sections of any abelian schemes over quasi-projective varieties over number fields.

 Let $S$ be a quasi-projective variety defined over a number field $K$,  $A/S$ an abelian scheme, and $L/A$ a line bundle which is relatively ample and symmetric in the sense that $[-1]^*L=L$. For a point $x \in A(\overline{K})$, we write
\[
h_\NT (x)
\]
for the associated Néron--Tate height. For each point $s\in S(\bar K)$,   we write
\[
h_{\mathrm{Fal}}(A_s)
\]
for the stable Faltings height of $A_s$.

Let $\nu: S\to A$ be a section. We write   $\nu_s=\nu(s)$. The main object of this paper is the relationship between $h_\NT(\nu_s)$ and $h_{\mathrm{Fal}}(A_s)$ in algebraic families.
 \begin{definition}[Bigness]\label{def:bigness-intro}
 	We say that   $\nu$ is big on $S$ if there are positive numbers $\epsilon, c\in \BR$, and a non-empty open subset $U\subset S$ such that for any $s\in U(\bar K)$, 
 	$$h_{\NT}(\nu_s)\geq \epsilon \hFal(A_s)-c.$$
 	 \end{definition}
	 One consequence of the bigness of $\nu$ is the height function
	 $$S(\bar K)\to \BR,\qquad s\mapsto h_{\NT}(\nu_s)$$
	 satisfies the Northcott property on the open subset $U$,  i.e.
	 for any $D, A\in \BR_{>0}$, the set 
		$$\{s\in U(\bar{K})~|~ [K(s) : K]<D, h_{\NT}(\nu_s)<A\}
		$$
    		is finite.

	This condition may be viewed as a quantitative strengthening of the generic non-torsion condition. Indeed, if $\nu$ is big on $S$, then $\nu_s$ is non-torsion for all transcendental $s\in S(\BC)$ outside a proper Zariski closed subset. Moreover, the height of $\nu_s$ grows at least linearly with the arithmetic complexity of the fiber.

\begin{remark}
	The precise normalization of $h_{\mathrm{Fal}}$ is irrelevant for the notion of bigness, since changing the normalization only modifies the constant $c$. \end{remark}

Let $S$ and $\nu$ be as above. Following the setup in the next section, we associate with $\nu$ a smallest abelian subscheme of the relative Jacobian containing the image of $\nu$ up to torsion translation.

\begin{definition}[The minimal abelian subscheme]
	\label{def:min-abelian}
	Let $A \to S$ be an abelian scheme and let $\nu : S \to A$ be a section. We define $A_{\min} \subseteq A$ to be the \emph{minimal abelian subscheme associated with $\nu$} if $A_{\min}$ is the smallest abelian subscheme of $A$ over $S$ for which there exists a section $\tau : S \to A$ of finite order such that
	$$
	\nu(S) \subseteq A_{\min} + \tau.
	$$
\end{definition}

The existence and uniqueness of $A_{\min}$ are standard in the relative setting. In the applications considered below, $A_{\min}$ will often be determined by explicit geometric decompositions of the Jacobian.

For every abelian subscheme $B \subseteq A_{\min}$, the quotient $A_{\min}/B$ is again an abelian scheme over $S$. Let
$$
q_B := \dim_S(A_{\min}/B).
$$
After choosing a suitable level structure, the quotient $A_{\min}/B$ defines a modular map
$$
\phi_B : S \longrightarrow \CS_{q_B},
$$
where $\CS_{q_B}$ denotes the Siegel moduli space of principally polarized abelian varieties of dimension $q_B$, or the appropriate finite-level cover if needed. The dimension of the image $\phi_B(S)$ measures the amount of modular variation of the quotient $A_{\min}/B$.
Let $\CA_{q_B}\to \CS_{q_B}$ be the universal family of abelian schemes. 
Now we have the composition of maps for every subscheme $B$: 
	$$\varphi_B: \quad S\overset {\nu-\tau}\lra A_{\min} \lra A_{\min}/B\lra \CA_{q_B}.$$

	Following the work of Gao \cite{Gao20}, Yuan--Zhang \cite{YZ21}, and Gao--Zhang \cite{GZ24}, the first main result of the paper gives a criterion for this bigness property in terms of modular variation of suitable abelian quotients. 
	\begin{theorem}[Criterion for bigness]
		\label{thm: criterion}
		Let \(A/S\) be an abelian scheme  over a quasi-projective variety defined over a number field, and let		\[
		\nu : S \longrightarrow A
		\]
		be the section. Let \(A_{\min}\) be the minimal abelian subscheme of \(A\) over \(S\) associated with \(\nu\). Then \(\nu\) is big on \(S\) if and only if for every abelian subscheme \(B \subseteq A_{\min}\) over $S$, one has
		\[
		\dim S \leq \dim \varphi_B(S) + \dim_S B.
		\]
	\end{theorem}
	
	For the proof, on the arithmetic side, we construct adelic line bundles whose associated height functions control the Néron--Tate height of \(\nu\). On the transcendental side, we use a generic non-degeneracy statement for Betti maps to convert modular variation into positivity.
	We then apply Theorem~\ref{thm: criterion} to two explicit families of genus \(4\) curves.
	
	The first family $\sC_\sN\to \sN\subset \BP^3$ is the locus of genus \(4\) curves admitting triple involutions:
	$$a(1+x^3y^3)+b(xy+x^2y^2)+c (x^2+xy^3)+d (y^2+x^3y)=0.$$
	The Jacobian admits an explicit decomposition, up to isogeny, into elliptic factors arising from quotient constructions. The canonical quadratic point has nontrivial projections to these factors, and this allows one to verify the dimension inequalities appearing in Theorem~\ref{thm: criterion}.
	
	\begin{theorem}[Triple involution locus]
		\label{thm:intro-triple}
		The canonical quadratic point \(\xi\) is big on  $\sN$. Moreover, there is an isogeny 
		$$\psi: \sJ_\sN\to \sE_0\times \sE_1\times \sE_2\times\sE_3$$
		into the product of four elliptic curves not isogenous to each other, such that $\xi$ has a non-zero projection on each of them.
	\end{theorem}
	
	Our second application concerns families containing a special CM curve with $\zeta_{15}$-action. In this setting, the simplicity of the Jacobian of the CM fiber, together with the existence of another fiber for which \(\xi\) is non-torsion, implies that the relevant abelian subscheme is as large as possible. This makes it possible to apply the criterion above in a particularly clean way.
	
	\begin{theorem}[A CM family]
		\label{thm:intro-cm}
		Let $S$ be a family of smooth genus 4 curves containing a special CM curve $C_0: y^5=x^3-1$ and another curve $C_1$ in $\sN$ such that $\xi_{C_1}$ is non-torsion. If  $\dim S \leq 4$, then $\xi$ is big on $S$.
	\end{theorem}
	
	These bigness results have direct arithmetic consequences. Since bigness gives a lower bound for the Néron--Tate height of $\xi_C$ in terms of the Faltings height of $C$, one obtains finiteness statements of Northcott type on suitable Zariski open subsets. Moreover, one deduces non-torsion of $\xi_C$ for curves corresponding to transcendental points of the parameter space.
	
	\begin{corollary}
		\label{cor:intro-consequences}
		On each of the families in Theorems~\ref{thm:intro-triple} and~\ref{thm:intro-cm}, there exists a Zariski open subset $U$ such that for any $D, A\in \BR$, the set 
		$$\{[C]\in U(\bar{\BQ})~|~ \deg C<D, \hFal(C)<A\}
		$$
		is finite.
	\end{corollary}
	
	The main theme of this paper is thus that the arithmetic of the canonical quadratic point is governed by the geometry of the ambient family and the variation of its abelian quotients. In particular, bigness provides a bridge between transcendental non-degeneracy, positivity in adelic Arakelov theory, and arithmetic finiteness questions for genus 4 curves.
	
	We conclude the introduction with a brief outline of the paper. In Section~\ref{sec:preliminaries}, we summarize the results in adelic line bundles by Yuan and Zhang \cite{YZ21} and Betti maps by Gao \cite{Gao20} needed in the proof of the fundamental criterion.   In Section~\ref{sec:criterion}, we prove the criterion for bigness. In Section~\ref{sec:geometry}, we study the geometry of a family of curves of genus 4 with triple involutions.
	Section~\ref{sec:triple} is devoted to the arithmetic of the family of genus \(4\) curves with triple involutions. In Section~\ref{sec:cm}, we study a CM family containing a curve with \(\zeta_{15}\)-action. Finally, in Section~\ref{sec:consequences}, we derive arithmetic consequences.

	\subsection*{Acknowledgments}
This work forms part of
the author's Ph.D. thesis.
	The author would like to express her sincere gratitude to her advisor
	Shou-Wu Zhang for his constant guidance, encouragement, and support
	throughout the development of this work. The author is also grateful to
	Ziyang Gao, Xinyi Yuan, and Hang Xue for their valuable comments and
	suggestions on earlier drafts of the manuscript.

	\addtocontents{toc}{\protect\setcounter{tocdepth}{0}}
	\section*{\textbf{Notations and conventions}}

	We conclude with a few conventions used throughout the paper.
	
	\begin{itemize}
		\item Unless specified, all varieties are defined over $\bar\BQ$.
		\item By a family of curves, we always mean a smooth proper morphism whose geometric fibers are smooth non-hyperelliptic curves. 
		\item The symbol $\sM$ is reserved for the moduli space of smooth genus 4 curves with smooth quadrics; the symbols $\sJ,\sC$ are reserved for the relative Jacobian and the universal curve of $\sM$.
		\item The notation \(h_{\mathrm{Fal}}(C)\) always refers to the stable Faltings height of \(\Jac(C)\).
		\item The notation \(\hat h(\nu(s))\) and \( h_{\NT}(\nu(s))\) refers to the Néron-Tate height on the Jacobian of the fiber \(\mathcal{C}_s\) with respect to its canonical principal polarization.
		
	\end{itemize}

	\addtocontents{toc}{\protect\setcounter{tocdepth}{1}}

	\section{Adelic line bundles and Betti maps}\label{sec:preliminaries}
	
	 We summarize the adelic line bundles and Betti maps needed in the proof of the fundamental criterion.

	\subsection{Adelic line bundles}
	
	The proof of the criterion for bigness relies on adelic line bundles on quasi-projective varieties and on the associated height functions. We only record the formal properties needed later and refer to the literature for the full theory.
	
	Let $S$ be a quasi-projective variety over a number field. An adelic line bundle on $S$ is a line bundle equipped with a compatible collection of metrics at all places. To such an object one can associate a height function on $S(\overline{K})$, well defined up to bounded functions. We will use the theory developed by Yuan and Zhang.
	
	Let $X\to S$ be a projective flat morphism of integral schemes. Let $f: X\to X$ be a morphism over $S$ and $P$ a $\BQ$-line bundle on $X$, relatively ample over $S$, such that $f^*P\cong qP$ for $q\in \BQ_{>1}$. We call $(X, f, P)$  a polarized algebraic dynamical system. One obtains adelic line bundles on $X$ whose restrictions to each fiber recover the canonical height associated to the principal polarization. Then we consider the  case that $X/S$ is an abelian scheme and $P/X$ is a relatively ample and symmetric line bundle.
	The precise construction and the required positivity statement will be given in Section~\ref{sec:criterion}. For the general theory of adelic line bundles and arithmetic positivity, we refer to \cite[2.5 Adelic line bundles, 6 Algebraic Dynamics]{YZ21}.

Now let $\nu: S\to X$ be a section. Then we have an adelic line bundle $\bar L:=\nu^* \bar P$ on $S$ and the following equality holds:
$$\hat h(\nu(x))=h_{\bar L}(x), \qquad x\in S(\bar K).$$
	Let $\wt L$ be the restriction of the image of $\bar{L}$ in $\widehat{\Pic}(S/\bar{K})$.
 	Applying \cite[Theorem 5.3.7(3), Theorem 5.3.8]{YZ21} we  conclude the following:
\begin{thm}[Yuan--Zhang \cite{YZ21}] \label{thm-YZ}The $\wt L$ is big if and only if  $\nu$ is big  on $S$, i.e. for any adelic line bundle $\bar M$ on $S$, there exist $c>0, \epsilon>0$ and a non-empty open subvariety $U$ of $S$ such that 
	$$h_{\bar{L}}(x)\geq \epsilon h_{\bar M}(x)-c,\qquad x\in U(\bar{\BQ}).$$
\end{thm}
	
	\subsection{Betti maps and modular variation}\label{betti}
The second input in the proof of the main criterion is transcendental. Over the complex numbers, an abelian scheme gives rise locally to a variation of Hodge structure, and sections of the abelian scheme determine multi-valued real-analytic maps in Betti coordinates. These are usually referred to as \emph{Betti maps}.

For the families considered in this paper, the relevant fact is that sufficient variation in moduli forces generic non-degeneracy of the associated Betti map. We use this in the form established by Gao \cite{Gao20}. Combined with the adelic positivity formalism, this yields the bridge between modular variation and lower bounds for Néron--Tate heights.

Let $S$ be a complex quasi-projective variety and $\pi: A\to S$ an abelian scheme. Then  we have local systems $\BV_\BZ$ and $\BV_\BR$ 
whose fibers are   $H_1(A_s, \BZ)$ and $H_1(A_s, \BR)$ ($s\in S(\BC)$). As real manifolds, the fibers $A_s(\BC)$ form a local system, and one has a natural identification:
	$$A(\BC)\iso \BV_\BR/\BV_\BZ.$$
The local horizontal sections form a foliation $\CF_\betti$ on $A$. More concretely, let $U\subset S$ be an open subset so that $\BV_\BZ|U$ is trivialized:
\begin{itemize}
	\item The lattice $\BV_\BZ|_U\cong \BZ^{2g}\times U$.
	\item Under this trivialisation 
	$$A_U(\BC)\cong (\BR^{2g}/\BZ^{2g})\times U.$$
	\item Projection to the torus factor: we have a map 
	$$b_U: A_U(\BC)\to \BR^{2g}/\BZ^{2g}$$
	which is real-analytic, and its fibers are copies of $U$ sitting horizontally in the total space. 
	These fibers are leaves of the foliation $\CF_\betti$. \end{itemize}

The section $\nu$ induces a map of tangent spaces: 
$$d\nu: TS\to \nu^*TA.$$
At a point $x\in A(\BC)$ lying over $s=\pi(x)\in S$ one has the direct sum decomposition
$$T_xA(\BC)=T_x \CF_\betti\oplus T_xA_s(\BC).$$

\begin{definition}[Betti map]
	The Betti map is the composition of the second  projection with the map $d\nu$:
	$$\nu_{\betti,s}: T_sS\stackrel{d\nu}{\to} T_{\nu(s)}A_s(\BC)\to T_{\nu(s)}A(\BC)_s,\qquad s\in S$$
\end{definition}

\begin{definition}[Betti rank]
	Let $S$ be a  subvariety of $\sM$ and $s\in S$. We call 
	$$\max_{s\in S(\BC)}\dim_{\nu_{\betti, s}}(T_sS)$$
	the Betti rank of $\nu$ and denote it by $r(\nu)$.
\end{definition}

By definition, the inequality $r(\nu)\geq 2\dim S$ holds if and only if $r(\nu)=2\dim S$.  Applying \cite[Theorem 1.1]{Gao20} with $l=\dim S$, the contrapositive of Gao's theorem gives the following criterion for the maximal Betti rank. 
 \begin{thm}[Gao \cite{Gao20}]\label{thm-Gao}
The equality $r(\nu)=2\dim S$ holds if and only if for any abelian subscheme $B$ of $A$ over $S$, we have $\dim S\leq \dim_{\varphi_B}(S)+\dim_SB$.
\end{thm}

	\section{A criterion for bigness}\label{sec:criterion}
	
	In this section, we prove the main criterion relating the bigness of a section $\nu$ for an abelian scheme $\pi: A\to S$ over a quasi-projective variety over a number field.  The argument has two parts. First, one associates with the section \(\nu\) an adelic line bundle on the base whose height function is given by the Néron--Tate height of \(\nu\). Second, one uses a generic non-degeneracy statement for Betti maps to show that the resulting adelic line bundle is positive precisely when the expected dimension inequalities hold.

\subsection{A dynamical construction}\label{dynamic process}
\newcommand{\can}{\mathrm{can}}
Let $\pi: A\to S$ be an abelian scheme over a quasi-projective variety defined over a number field $K$ and $P$ a symmetric and relatively ample line bundle. 
Consider the polarized dynamical system $(A/S, [2], P)$, where $[2]: A\to A$ is the multiplication by 2 map and $[2]^*P=4P$ by the symmetry of $P$.
By the construction \cite[Theorem 6.1.1]{YZ21} and   \cite[Theorem 6.1.2]{YZ21}  we have an adelic line bundle $\bar P$ on $A$ that satisfies
\begin{itemize}
	\item $[2]^*\bar P=4\bar P$.
	\item $\bar P$  is nef. 
\end{itemize} 

Let $\nu: S\to A$ be a section of $\pi$. 
Pulling back by $\nu$, we obtain an adelic line bundle on $S$: 
$$\bar L:=\nu^*\bar P.$$
Thus,$$h_{\bar L}(s)=h_{\bar P}(\nu(s)), \qquad\forall s\in S(\bar K).$$
Let $\wt L$ be the restriction of the image of $\bar{L}$ in $\widehat{\Pic}(S/\bar{K})$.
 Thus, the bigness of $\nu$ on $S$ is equivalent to the bigness of $\wt{L}$ on $S$.

\subsection{Volume formula}
Now we embed $\bar K\incl \BC$.  On $S(\BC)$ and $L(\BC)$, recall that we have the Chern curvature locally given  by the (1, 1)-form 
$$c_1(\bar{L})=\frac{i}{2\pi}\partial\bar{\partial}\log \norm{\ell}^2,$$
where $\ell$ is a trivialization of $L$ on an open subset $U(\BC)\incl S(\BC)$ such that $L|_U\cong\CO_S|_U$.
The same proof of \cite[Proposition 6.3]{GZ24} shows that 
\begin{thm}[Gao--Zhang \cite{GZ24}]\label{lem1}
	$$\widehat{\vol}(\tilde{L})=\int_{S(\BC)}c_1(\bar{L})^{\dim S}.$$
\end{thm}
\begin{proof}
	By the flatness of extension $K\subset \BC$, we have $\widehat{\vol}(\tilde{L}_K)=\widehat{\vol}(\tilde{L}_\BC)$. If the adelic line bundle $\bar{L}$ on $S$ is represented by $(L, (\CS_i, \bar L_i, \ell_i)_{i\geq 1})$, by \cite[Theorem 5.2.1]{YZ21}
	$$\widehat{\vol}(\wt{L})=\lim_{i\to\infty}\widehat{\vol}(S_i, L_i)$$
	Consider the integral projective model $\CA\to \CS$ of $\pi$,  in the construction \ref{dynamic process} and by  \cite[Theorem 6.1.1]{YZ21} we have a sequence of line bundles $(P, (A_i, \bar P_i, \ell_i))$  with limit $(A, \bar P)$ with respect to the dynamical system $(A, [2], P)$.
Moreover, $\bar P$ is nef. More precisely, there  is an ample Hermitian line bundle $\bar{N}$  on $A$ such that $\bar{\sP}^\Delta_i+4^{-i}\pi^*\bar{N}$ is nef for each $i$.

	For any $n\in \BN$, we have an action of $A[n]$ on $A$ by translating  torsion points: 
	$$(m, p): A[n]\times A\to A\times A,\qquad (t, x)\mapsto (x+t, x).$$
	Where $p$ is the projection on the second factor. By  the theorem of square, $N_pm^*P=(P)^{n^{2g}}$. Thus, we obtain an adelic metrized line bundle 
	$$\bar P_{i,n}:=n^{-2g}N_pm^*\bar P_i.$$
	This bundle is realized on some projective model $A_{i, n}$ of $A$,   with connection morphism $\ell_{i, n}: P\to P_{i, n, A}.$ Moreover $\div(\ell_{i, n})$ is in fact bounded by $\div(\ell_i)$. Over $\BC$ the curvature form $c_1(\bar P(\BC))$ is obtained from the curvature form $c_1(\bar P_i(\BC))$ by taking the average over $n$-torsion points. It  follows that these forms converge to $c_1(\bar{P})$ uniformly in any compact subset of $A(\BC)$. These bundles also induce a double sequence of model line bundles $(S_i, \bar L_{i,n})$ of $(S, \bar{L})$  so that they converge to $(S, \bar{L})$ as $i\to \infty$, and that the forms $c_1(\bar{L}_{i, n}(\BC))$ uniformly converge to $c_1(\bar{L}(\BC))$ as $n\to\infty$.

	More precisely, let \( \Omega_i \) be an increasing sequence of relatively compact open subsets of \( S(\mathbb{C}) \) such that \( S(\mathbb{C}) = \bigcup \Omega_i \), and \( \epsilon_i \) be a decreasing sequence of positive numbers converging to 0 so that on \( \Omega_i \),
	\[
	c_1({L}_{\mathbb{C}}) \leq \epsilon_i^{-1} c_1({N}_{\mathbb{C}}).
	\]
Then for each \( i \), choose \( n_i \) so that
	\begin{equation} \label{eq:approx}
		- \epsilon_i^d c_1({N}_{\mathbb{C}}) \leq c_1({L}_{i,n_i,\mathbb{C}}) - c_1({L}_{\mathbb{C}}) \leq \epsilon_i^d c_1({N}_{\mathbb{C}})
	\end{equation}
	as Hermitian forms on the tangent bundle over \( \Omega_i \). Set \( {L}_i := {L}_{i,n_i} \).
	
	Now apply Demailly’s Morse inequality \cite{Dem91} to \( {L}_i \) on \( S_i(\mathbb{C}) \). For each \( q \in \mathbb{N} \), let \( S_{i,q} \) denote the subset of points where \( c_1({L}_i) \) has \( q \) negative eigenvalues. Then let $d=\dim S$,
	\[
	h^q(k{L}_i) \leq \frac{k^d}{d!} \left| \int_{S_{i,q}} c_1({L}_i)^d \right| + o(k^d),
	\quad
	\sum_q (-1)^q h^q(k{L}_i) = \frac{k^d}{d!} \int_{S(\mathbb{C})} c_1({L}_i)^d + o(k^d).
	\]
	It follows that
	\[
	\left| \widehat{\mathrm{vol}}({L}_i) - \int_{S(\mathbb{C})} c_1({L}_i)^d \right|
	\leq \sum_{q > 0} \left| \int_{S_{i,q}} c_1({L}_i)^d \right|.
	\]
	By \cite[Theorem 5.4.4]{YZ21},
	\[
	\lim_{i \to \infty} \int_{S(\mathbb{C})} c_1({L}_i)^d = \int_{S(\mathbb{C})} c_1({L})^d.
	\]
	It remains to show \( \int_{S_{i,q}} c_1({L}_i)^d \to 0 \) for each \( q > 0 \), which is done by estimating the integral over \( \Omega_i \cap S_{i,q} \) and \( S_{i,q} \setminus \Omega_i \) respectively, using inequality \eqref{eq:approx} and constructing cut-off functions. We omit further technicalities here.

We divide the integral  into two parts, over \( \Omega_i \cap S_{i,q} \) and over its complement \( S_{i,q} \setminus \Omega_i \).

From inequality \ref{eq:approx}, we know:
\[
- \epsilon_i^d c_1({N}_{\mathbb{C}}) \leq c_1({L}_i) \leq (\epsilon_i^{-1} + \epsilon_i^d) c_1({N}_{\mathbb{C}}).
\]
Hence on \( \Omega_i \cap S_{i,q} \), all eigenvalues of \( c_1({L}_i) \) are bounded above by \( \epsilon_i^{-1} + \epsilon_i^d \), and there is at least one negative eigenvalue with absolute value bounded by \( \epsilon_i^d \).
Thus, the top-degree form \( |c_1({L}_i)^d| \) is bounded by
\[
\epsilon_i^d (\epsilon_i^{-1} + \epsilon_i^d)^{d-1} c_1({N})^d = \epsilon_i (1 + \epsilon_i^{d+1})^{d-1} c_1({N})^d.
\]
It follows that:
\[
\int_{S_{i,q} \cap \Omega_i} |c_1({L}_i)^d| = O(\epsilon_i).
\]

Now, we handle the integral over \( S_{i,q} \setminus \Omega_i \). Let \( \Omega_i' \subset \Omega_i \) be a sequence of compact subsets with \( \bigcup \Omega_i' = S(\mathbb{C}) \). Then there is an increasing sequence of continuous functions \( f_i \) such that:
\begin{itemize}
\item 
 \( f_i(x) = 1 \) on \( S(\mathbb{C}) \setminus \Omega_i \),
\item \( f_i = 0 \) on \( \Omega_i' \).
\end{itemize}
Then for any \( i \geq j \), we have
\[
\int_{S(\mathbb{C}) \setminus \Omega_i} c_1({L}_i )^d 
\leq \int_{S(\mathbb{C})} f_i \cdot c_1({L}_i)^d 
\leq \int_{S(\mathbb{C})} f_j \cdot c_1({L}_i )^d.
\]
Fix \( j \), then as \( i \to \infty \),
\[
\limsup_{i \to \infty} \int_{S(\mathbb{C}) \setminus \Omega_i} c_1({L}_i )^d
\leq \int_{S(\mathbb{C})} f_j \cdot c_1({L})^d 
= \int_{S(\mathbb{C}) \setminus \Omega_j'} c_1({L})^d.
\]
Letting \( j \to \infty \), we conclude:
\[
\lim_{i \to \infty} \int_{S(\mathbb{C}) \setminus \Omega_i} c_1({L}_i )^d = 0.
\]
\end{proof}
So we have the following
\begin{cor}\label{lem2}
	The $\nu$ is big on $S$ if and only  if $c_1(\bar{L})$ is not identical to 0 on $S(\BC)$.
\end{cor}

\subsection{Proof of the theorem \ref{thm: criterion}}

To prove the bigness of $\nu$, we want to relate the volume form $c_1(\bar{\CL})^{\dim S}$  to the non-degeneracy of the Betti map. We refer to the definitions of Betti rank and Betti map in \ref{betti}. Thus, we have the equivalence.
\begin{lem}\label{lem3}
		For any $s\in S(\BC)$, the following are equivalent:
	\begin{itemize}
		\item $c_1(\bar{\CL})_s^{\dim S}\neq 0$.
		\item The Betti map is injective, i.e. $\dim_{\nu_{\betti, s}}(T_s S)=\dim S$.
	\end{itemize}
\end{lem}

\begin{proof}
	Assume that $c_1(\bar{\CL})_s^{\dim S}= 0$ for an $s\in S(\BC)$. Then there exists $0\neq u\in T_sS$ with $\nu_{\betti, s}(u)=0$. Thus $\ker \nu_{\betti, s}\neq 0$. Therefore, (2) fails. 
	
	Conversely, assume that the Betti map is not injective. Then there is $0\neq u\in T_sS$ such that $\nu_{\betti, s}=0$. Thus $c_1(\bar{\CL})_s(u, \bar{u})=0$.  Thus $u$ is an eigenvector of the Hermitian matrix defining the 1-1 form $c_1(\bar{\CL})$ with eigenvalue 0. Hence, the determinant of this matrix is 0, so $c_1(\bar{\CL})^{\dim S}=0$ at $s$. Thus (1) fails.
\end{proof}

\begin{proof}[Proof of  Theorem \ref{thm: criterion}]

We will use the same terminology as in the beginning of this chapter. 
We follow the strategies used in the papers \cite{GZ24} and \cite{Gao20}. 
\begin{itemize}
	\item The \cite{GZ24}  relates the bigness of $\nu$ on $S$ with the non-degeneracy of the 1-1 form $c_1(\bar{\CL})$ and the Betti rank of $\nu$.
	\item The  \cite{Gao20} relates the inequality $\dim S\leq \dim_{\varphi_B}(S)+\dim B$ with the Betti rank of $\nu$. 
\end{itemize}

Step 1: By Theorem \ref{lem1} we have
$$\widehat{\vol}(\tilde{\CL})=\int_{S(\BC)}c_1(\bar{\CL})^{\dim S}.$$
Then by Theorem \ref{thm-YZ},  the bigness of $\nu$ is equivalent to $c_1(\bar{\CL})\nequiv0$.

Step 2: By Lemma \ref{lem3} $c_1(\bar{\CL})\nequiv0$ if and only if the Betti rank $r(\nu)=2\dim S$. 

Step 3: By Theorem \ref{thm-Gao}, 
$r(\nu)=2\dim S$ if and only if for any abelian subscheme $B$ of $A$ over $S$, we have $\dim S\leq \dim_{\varphi_B}(S)+\dim_SB$.
So we are done.
\end{proof}

\section{Geometry of curves of genus four with  triple involutions}
\label{sec:geometry}

In this section, we construct a  $3$-dimensional family of genus $4$ curves admitting triple involutions. The main feature of this family is that the Jacobian admits an explicit decomposition, up to isogeny, into elliptic factors arising from quotient constructions. This makes it possible to verify the dimension inequalities in Theorem~\ref{thm: criterion} by hand.

\subsection{Definition of the family}

Let $\BP^1$ be equipped with the involutions
$$
\omega_1 : x \mapsto -x,
\qquad
\omega_2 : x \mapsto x^{-1},
\qquad
\omega_3 := \omega_1 \circ \omega_2 : x \mapsto -x^{-1}.
$$
These generate a subgroup $W$ isomorphic to $(\ZZ/2\ZZ)^2$ inside $\Aut(\BP^1)$. 
Let $\tilde{W}$ denote the normalizer of $W$ in $\Aut(\BP^1)$. Then $\wt W/W$ permutes $\omega_i$'s, thus induces a map 
$$\wt W/W\lra S_3.$$

\begin{prop}\label{S_3}
The above map is an isomorphism   $\tilde{W}/W\iso S_3$.
In fact, $$\tilde{W}\cong W\rtimes S_3\cong S_4.$$
\end{prop}
\begin{proof}
Since every $g\in\tilde{W}$ acts on $W$ by conjugation,  there is a homomorphism
$$\psi: \tilde{W}\to \Aut(W),\qquad g\mapsto (\psi(g): \omega_i\mapsto g\omega_i g^{-1} ).$$	
Since $W\cong (\BZ/2\BZ)^2$, we have 
$$\Aut(W)\cong \GL_2(\BF_2)\cong S_3.$$
Thus $\tilde{W}/W$ is a subgroup of $S_3$. 
\begin{itemize}
	\item The kernel is the centralizer of $W$. An element $g\in \Aut(\BP^1)$ commutes with $\omega_1$ forces $g$ to be one of the form 
	$$g: x\mapsto \lambda x \qquad \text{or}\qquad g: x\mapsto \frac{\lambda}{x}.$$
	If $g$ also commutes with $\omega_2$, it forces $\lambda=\pm1.$ Hence $\ker\psi=W$. The conjugation map induces an injection
	$$\tilde{W}/W\incl \Aut(W)\cong S_3.$$
	\item Over $\bar{\BQ}$ the injection map is surjective. Since $S_3$ is generated by a transposition (12) and a 3-cycle (123). We find matrices in $\tilde{W}$ that induce these permutations: For the transposition (12) we need to find $g\in \tilde{W}$ such that 
	$$g\omega_1=\omega_2g,\qquad g\omega_2=\omega_1g,\qquad g\omega_3=\omega_3g.$$
	By calculation $g=\begin{pmatrix}
		1&1\\
		1&-1
	\end{pmatrix};$ for the 3-cycle (123), let $i=\sqrt{-1}$ and $g=\begin{pmatrix}
	1&-i\\
	1&i
\end{pmatrix}$,
then $$g\omega_2(x)=-1/g(x)=\omega_3g,\qquad g\omega_3(x)=-g(x)=\omega_1 g,\qquad g\omega_1(x)=1/g(x)=\omega_2 g.$$
\end{itemize}
\end{proof}

 By letting these involutions in $W\times W$ act on both factors of $\BP^1 \times \BP^1$, one obtains induced involutions on the space $\sM$ of curves of bidegree $(3,3)$.

We consider the diagonal subgroup $\Sigma:=\Delta W$ of $W^2$ of involutions 
$$\sigma_i=(\omega_i, \omega_i),\qquad i=1, 2, 3$$
 on $\BP^1\times \BP^1$, where $\sigma_3=\sigma_1\cdot \sigma_2$. 
 Notice that $\sigma_1, \sigma_2$ generate a Klein four subgroup $\Sigma=\Delta W\subset W^2$:
 $$\Sigma=\{1, \sigma_1, \sigma_2, \sigma_1\sigma_2\}.$$
 Therefore, we call this a ``triple involution".

We consider the moduli space $\sM^\Sigma$ of smooth curves
$$
C \subset \BP^1 \times \BP^1
$$
of bidegree (3, 3) that are preserved by a prescribed triple  involution, namely the subgroup $\Sigma$. Then $\sM^\Sigma$ has  actions by $W^2/\Sigma$, and $\wt \Sigma/\Sigma$, where $\wt \Sigma =\Delta \wt W$.

\begin{prop}\label{prop-lm} The variety $\sM^\Sigma $ has four connected components $\sM^\Sigma _{\lambda, \mu}$ indexed  $(\lambda, \mu)\in \{\pm 1\}^2$.
 Each component $\sM^\Sigma_{\lambda, \mu}$ is isomorphic to an open subset of $\BP^3$ with universal curve $C_{\lambda, \mu}$ given as follows:
 $$\sC_{1, \pm 1}: a(1\pm x^3y^3)+b(xy\pm x^2y^2)+c(x^2\pm xy^3)+d(y^2\pm x^3y)=0,$$
		$$\sC_{-1, \pm 1}: a(x\pm x^2y^3)+b(y\pm x^3y^2)+c(x^3\pm y^3)+d(y\pm x^3y^2)=0,$$
		where $[a, b, c, d]$ are the homogeneous coordinates of $\BP^3$.
Moreover, the group $W^2/\Sigma$ acts transitively on this set of connected components, and $\wt \Sigma/\Sigma$ stabilizes $\sN_{1,1}$, and permuting other three components.
\end{prop}
\begin{proof} For any field $k$, 
	consider the equation defining $C\in \sM^\Sigma(k)$:
	$$0=f(x, y)=\sum_{0\leq i, j\leq 3}a_{ij}x^i y^j, \qquad a_{ij}\in k$$
	Thus $\sigma_i$ acts on $C$ as follows: 
	
	$$\sigma_1: (x, y)\mapsto (-x, -y);$$
	$$\sigma_2: (x, y)\mapsto (1/x, 1/y);$$
	$$\sigma_3: (x, y)\mapsto (-1/x, -1/y);$$
	
	\begin{itemize}
		\item The $\sigma_1$ stabilizes $C$ if and only if there is a $\lambda\in \BC^\times$ such that 
		$$f(x, y)=\lambda f(-x, -y).$$ 
		Applying $\sigma_1$ one more time we have $f(x,y)=\lambda^2 f(x, y)$. Thus $\lambda=\pm1$.
		Thus, when $\lambda=1$,  $a_{i j}= 0$ if $i+j$ is even, and when $\lambda =-1$, $a_{ij}=0$ if $i+j$ is odd.
				\item  The $\sigma_2$ stabilizes $C$ if and only if 
		$$f(x, y)=\mu f(1/x, 1/y)x^3y^3, \qquad \mu\in \BC^\times.$$ 
		Applying $\sigma_2$ one more time we have $f(x,y)=\mu^2 f(x, y)$. Thus $\mu=\pm1$.
		Then $a_{ij}=\mu a_{3-i, 3-j}$.
		\item The $\sigma_3$ stabilizes $C$ if and only if  $\lambda\mu=\pm1$. Let $a, b, c, d$ be constants. Thus we have four types of equations corresponding to $(\lambda, \mu)\in \{\pm 1\}^2$ with equation given as follows: 		$$C_{1, \pm 1}: a(1\pm x^3y^3)+b(xy\pm x^2y^2)+c(x^2\pm xy^3)+d(y^2\pm x^3y)=0,$$
		$$C_{-1, \pm 1}: a(x\pm x^2y^3)+b(y\pm x^3y^2)+c(x^3\pm y^3)+d(xy^2\pm x^2y)=0.$$
		\end{itemize}
		
These show that $\sM^\Sigma$ has four components as in the statement of the proposition. 
	These families can be translated into each other by the subgroup $ W^2/\Sigma$.
	
	In fact, since every $\sigma_i$ stabilizes $C_{1, \pm 1}, C_{-1, \pm 1}$, we need only to show that these families can be translated by  $\sigma_i^1:=(\omega_i, \Id), i=1, 2.$ Denote the action of $\sigma_i^1$ on $C$ by $\sigma_i^1\cdot C$. By calculation we have
	$$\sigma_1^1\cdot C_{1, 1}=C_{1, -1},\qquad \sigma_1^1\cdot C_{-1,  1}=C_{-1,  -1};$$
	$$\sigma_2^1\cdot C_{1, \pm1}=C_{-1, \pm1},\qquad \sigma_2^1\cdot C_{-1,  \pm1}=C_{1,  \pm1}.$$
	
	The last statement about the action of $\wt\Sigma/\Sigma$ is clear. 
\end{proof}

Let $\sN$ denote $\sM^\Sigma_{1, 1}$, then $\sN$ is isomorphic to an open subset of $\BP^3$ over $\BQ$. Let $k$ denote the function field of $\BP^3$:
namely $k=\BQ(b/a, c/a, d/a)$. 

Let
$$
\pi : \sC:=\sC_{1,1}\longrightarrow \sN
$$
be the universal family, and let
$$
\sJ \longrightarrow \sN
$$
be the relative Jacobian. As in the previous sections, the canonical quadratic point determines a section
$$
\xi : \sN \longrightarrow \sJ.
$$

One of our goals in this paper is to prove that  the canonical quadratic point $\xi$ is big on $\sN$.

\subsection{Quotient curves}

The involutions on the fibers of $\pi$ give rise to several quotient curves. These quotients are the basic geometric input for the decomposition of the Jacobian.
\begin{lemma}\label{A_2}
	Let $C$ be the generic fiber of $\pi$. Let 
	$$
	X_i := C/\sigma_i, \qquad i=1,2,3.
	$$
	Then each $X_i$ is a hyperelliptic curve of genus $2$.
\end{lemma}
\begin{proof}
	Let
	$$
	f_i:C\longrightarrow X_i:=C/\sigma_i
	$$
	be the quotient map. Since $\sigma_i$ is an involution, $f_i$ has degree $2$. By the Riemann-Hurwitz formula,
	$$
	2g(C)-2 = 2\bigl(2g(X_i)-2\bigr)+\sum_{P}(e_p-1)
	$$
	where $P$ runs over the ramification points, in this case they are the fixed points.
	Since $g(C)=4$ and $\sigma_i$ has exactly $2$ fixed points, we obtain
$$
	6=2\bigl(2g(X_i)-2\bigr)+2.
$$
	Hence
$$
	g(X_i)=2.
$$
	Therefore $X_i$ is a  hyperelliptic curve of genus 2.
\end{proof}

More precisely,  recall that $C$ is defined by the  affine equation
$$F(x, y):=a(1+x^3y^3)+b(xy+x^2y^2)+c(x^2+xy^3)+d(y^2+x^3y)=0.$$
An affine  equation for $X_1$ is obtained by the substitutions
$u=xy,  v=x^2$:
\begin{eqnarray}
	a\left(1+u^{3}\right)+b\left(u+u^{2}\right)+c\left(v+\frac{u^{3}}{v}\right)+d\left(u^{2} / v+u v\right)=0 
\end{eqnarray}
Multiplying both sides by $v$, we have 
$$a \left(\left(1+u^{3}\right)+b\left(u+u^{2}\right)\right) v+(c+d u) v^{2}+\left(u^{3} c+d u^{2}\right)=0  $$
This is a quadratic function of $v$. Set 
$$w=2(c+d u) v+a(1+u^3)+b(u+u^{2}) =c(x^{2}-x y^{3})+d  (x^{3} y-y^{2}),$$
we have an affine equation for $X_1$:
\begin{align*}w^2=&\left(a\left(1+u^{3}\right)+b u(1+u)\right)^{2}-4(c+d u)(c u+d) u^{2}\\
=&a^2(u^6+1) + 2ab(u^5+u) + (b^2+2ab-4dc)(u^4+u^2)
+ (2a^2+2b^2-4c^2-4d^2)u^3
\end{align*}

Because the involutions $\sigma_1,\sigma_2,\sigma_3$ generate a Klein four subgroup, for each $i$, the involution $\sigma_j$ for $j\neq i$ descends to an involution on $X_i$. This descended involution is different from the hyperelliptic involution of $X_i$. Thus $X_i$ carries a non-hyperelliptic involution $\tau_i$ and a hyperelliptic involution $\iota_i$.

	Let $\sK$ denote the moduli space of pairs $(X, \tau)$ where $X$ is a genus $2$ hyperelliptic curve, and $\tau\in\Aut (X)$
	is a non-hyperelliptic involution. Then, for each $i$, we have defined a map
	$$
	\pi_i : \sN \longrightarrow \sK, \qquad [C] \longmapsto [X_i, \tau_i].
	$$

	\begin{lem}\label{lem-K}
	The moduli space $\sK$ is two-dimensional with universal hyperelliptic curve and involution given by the equation:
	$$
y^2 = A(x^6+1) + B(x^5+x) + C(x^4+x^2) + Dx^3, \qquad \tau (x, y)=(1/x,  y/x^3).
$$
This equation is unique up to automorphism 
$$x\mapsto \frac{ax+b}{bx+a}, \qquad a, b\in k, a^2\ne b^2.$$
\end{lem}
\begin{proof}
Let $k$ be a field of characteristic $0$ and $(X, \tau)$ a point of $ \sK (k)$. Let $\pi: X\lra \BP^1$ be a double cover defined by the hyperelliptic involution. 
Then $\pi$ is  defined by the space of sections $\Gamma (X, \Omega _X^1)$. As $\tau ^*\Omega^1 _X\cong\Omega^1_X$, we see that $\tau$ induces an involution on $\BP^1$ denoted by $\sigma$ so that $\sigma\circ \pi=\pi\circ \tau$. This shows that the ramification points of the covering $\pi$ are invariant under $\sigma$. After conjugating by an automorphism of $\BP^1$, we may assume that $\sigma$ is given by $x\mapsto 1/x$. This coordinate is unique up to isomorphism
$$x\mapsto \frac {ax+b}{bx+a}.$$Thus, the equation of $X$ and the involution must have the form
$$
y^2 = A(x^6+1) + B(x^5+x) + C(x^4+x^2) + Dx^3, \qquad \tau (x, y)=(1/x, \pm y/x^3).
$$
The two involutions can be interchanged by the substitution $x\mapsto -x$. This proves the first statement in the Theorem. For the second statement, we notice that the equation is unique up to substitution
$$(x, y)\mapsto \left(\frac {ax+b}{bx+a}, \frac y{(bx+a)^3}\right).$$
\end{proof}

\begin{prop}\label{lem:N to K}
		The above map $\pi_i$	is dominant.
\end{prop}

\begin{proof}
Since the quotient $X_1$ has the equation:
$$
y^2 = a^2(x^6+1) + 2ab(x^5+x) + (b^2+2ab-4dc)(x^4+x^2)
+ (2a^2+2b^2-4c^2-4d^2)x^3
$$
It suffices to show that for any general $A, B, C, D\in \bar \BQ$, there are $a, b, c, d\in \bar \BQ$
such that 
$$A=a^2, \quad B=2ab, \quad C=(b^2+2ab-4dc), \quad D=(2a^2+2b^2-4c^2-4d^2).$$
This follows immediately.
\end{proof}

\begin{prop}\label{lem:K to A1A1}
For a point $[X]\in \sK$, let $\iota$ be the hyperelliptic involution and let $\tau$ be the given non-hyperelliptic involution. Then the two quotients
	\[
	E_+:=X/\langle\tau\rangle,\qquad E_{-}:=X/\langle \iota_X\circ\tau\rangle
	\]
	are elliptic curves. The induced map
$$
	f: \sK\to \BA^1\times\BA^1,\qquad [X]\mapsto \bigl([E_+],[E_-]\bigr)
$$
	is surjective.
\end{prop}

\begin{proof}
To prove that $f$ is dominant, it suffices to show it is surjective over $\BC$. 
	Let $X$ be a genus $2$ curve with a non-hyperelliptic involution $\tau$. Since $X$ is hyperelliptic, it admits a unique hyperelliptic involution $\iota$. The involutions $\tau$ and $\iota\circ\tau$ are distinct involutions on $X$, and neither is equal to $\iota$.	By Lemma \ref{lem-K}, 
	we write $X$ in terms of its 6 ramification points 
	on $\BP^1(\BC)$ stable under involution $x\mapsto x^{-1}$.
	$$y^2=\prod _{i=1}^3 (x-\rho _i)(x-\rho _i^{-1})=\prod _{i=1}^3 (x^2-\alpha _ix+1).$$
	where $\alpha _i=\rho _i+\rho _i^{-1}$.	
		The invariants in $k(X)$ under these involutions are generated by 
	$$u:=x+1/x, \qquad v:=y(x^{-1}\pm x^{-2}).$$
	Using these invariants, we have the equation of two quotient elliptic curves:
	$$E_{\pm}: v^2=x^3(x^{-1}\pm x^{-2})^2\prod _{i=1}^3 (x^2-\alpha _ix+1)=(u\pm 2)\prod_{i=1}^3(u-\alpha _i).$$
	
	Now let $E_1, E_2$ be two elliptic curves represented in Legendre form:
	$$E_i: y^2=x(x-1)(x-\lambda _i), \qquad \lambda _i\in \BC\setminus\{0, 1\}.$$
	We can always assume that $\lambda _1\ne \lambda _2$ even in the case $E_1=E_2$ as we can change $\lambda _i$ to $1/\lambda _i$, $1-\lambda _i$, etc. Now choose an element $g\in \SL_2(\BR)$ so that 
	$g\lambda _1=2$ and $ g\lambda _2=-2$.
	Then $E_1\simeq E_+, E_2\simeq E_-$ with 
	$$\alpha _1=g(0), \qquad \alpha _2=g(1), \qquad \alpha _3=g (\infty).$$
\end{proof}

\subsection{Decomposition of the Jacobian}

For a curve $[C]\in \sN$, we have  3 quotients $X_i=C/\sigma _i$ of genus $2$, and further 6 quotients of genus $1$: $X_i/\tau_i$, $X_i/\tau_i\iota_i$.
The quotient $X_i/\tau_i=C/(\sigma_1,\sigma_2)$ does not depend on $i$. Thus we denote four elliptic curves:
$$E_0=C/(\sigma _1, \sigma _2)=X_i/\tau_i, \qquad E_1=X_i/\tau_i\iota _i, \quad (i=1,2, 3.)$$
In this way, we obtain an isogeny
$$\psi: \Jac(C)\to E_0\times E_1\times E_2\times E_3.$$
The $j$-invariants induce a map:
	$$J: \sN\to \BA^4, \qquad [C]\mapsto (j(E_0), j(E_1), j(E_2), j(E_3)).$$

\begin{prop}
	\label{prop:jacobian-decomposition-N}
	The map $J$ is finite. Moreover, 
	 for general $C \in N$, the elliptic curves $E_0,E_1,E_2,E_3$ are pairwise non-isogenous.
\end{prop}

\begin{proof}
	Fix $(j_0, j_1, j_2, j_3)\in \BC^4$ and let $E_i$ be elliptic curves with $j$-invariant $j_i$.
	Notice that the dual to $\psi$ defines $\Jac (C)$ as a quotient of $E_0\times E_1\times E_2\times E_3$ degree bounded by $4^4=256$.
	Thus, there are only finitely many possibilities for  $\Jac (C)$, namely a   finite list of principally polarized abelian varieties $A_1, \cdots, A_N$.
	By Torelli's theorem, the isomorphism classes of $C$ lie in a finite list of curves $X_1, \cdots, X_M$.
	It remains to prove that for each curve $X$ over $\BC$ of genus $4$, there are finitely many embeddings  $f: X\lra \BP^1\times \BP^1$
	so that the image is defined by an equation of degree $(3, 3)$ and stable under $\sigma_1, \sigma_2$.
	
	Notice the embedding $f$ is induced by a canonical embedding $C\lra \BP^3$. Thus, all embeddings are in one orbit under the composition
	$\alpha \circ f$ with 
	$$\alpha \in \Aut (\BP^1\times \BP^1)=\Aut (\BP^1)^2\rtimes (\BZ/2\BZ)$$
	where $1\in \BZ/2\BZ$ acts on $\BP^1\times \BP^1$ by switching  two factors which we denote as $\epsilon$.

	Also notice that for each such  embedding $f$, the $\sigma _1, \sigma _2$ induce embeddings $\varphi: (\BZ/2\BZ)^2\lra \Aut (C)$.
	Since $\Aut (C)$ is finite, there are only finitely many such embeddings $\varphi$. Thus, we can fix such $\varphi$ and consider the embeddings 
	$f$ so that $\varphi(1, 0)=\sigma _1$ and $\varphi(0,1)=\sigma_2$.

	For each pair $(X, \varphi)$ if we fixed one embedding $f_0$ as above, all other embeddings are given by composition $f_\alpha:=\alpha\circ f$ with 
	$\alpha\in \Aut (\BP^1\times \BP^1)$ commuting  with $\sigma _1$ and $\sigma _2$. All such $\alpha$ form a subgroup $G$ of $\Aut (\BP^1\times \BP^1)$. We need to show that $G$ is finite. Of course $\epsilon \in G$.
	Thus $G$ is generated by $W$  and $\epsilon$.
This proves the first part of the proposition.

		By the first part, the image of $J$ is a divisor of $\BA^4$. Thus, it is defined by a single irreducible polynomial equation 
	$$J(\sN): P(j_0, j_1, j_2, j_3)=0.$$

By \ref{S_3} and \ref{prop-lm},  we can conclude that  $\wt\Sigma/\Sigma\simeq S_3$ acts on $\sN$ which permutes $\sigma_1, \sigma_2, \sigma_3$. 
	Thus $S_3$ also acts on $J(\sN)$ which permutes $j_1, j_2, j_3$.
	This shows that $P(j_0, j_1, j_2, j_3)$ is symmetric in $j_1, j_2, j_3$.
	
	If for some pair $m<n$, $E_m$ and $E_n$ are isogenous to each other with degree $N$, then $j_m$ and $j_n$ satisfy one equation $\Phi_N(j_m, j_n)=0$. Then $P$ must be a scalar multiple of $\Phi_N$. Thus the polynomial $P$ has only two variables $j_m$ and $j_n$: $P(j_m, j_n)=0$. As none of $j_n$ is constant, a contradiction to the symmetric property of $P$ as above.
\end{proof}

This decomposition is the key reason the criterion in Section~\ref{sec:criterion} becomes effective on $\sN$: any abelian subvariety of $\sJ$ can be analyzed, up to isogeny, in terms of these elliptic factors.

\section{Arithmetic of curves of genus four with  triple involutions}\label{sec:triple}

In this section, we apply Theorem~\ref{thm: criterion} to the $3$-dimensional family of genus $4$ curves admitting triple involutions. 

\subsection{The projection of the canonical quadratic point is non-torsion}

We now study the image of the canonical quadratic point under the projections to the elliptic factors.
Let
$$
p_i : \Jac(C) \longrightarrow E_i
\qquad (i=0,1,2,3)
$$
denote the projections induced by the isogeny of Proposition~\ref{prop:jacobian-decomposition-N}. To show the canonical quadratic point  $\xi$ is non-torsion, it suffices to show that the point $\xi_C$ has nontrivial image in each factor.

\begin{prop}
		\label{prop:nonzero-projections}
	For a generic curve $C \in \sN$, one has
	$$
	p_i(\xi_C) \notin E_{i, \tor}
	\qquad\text{for all } i=0,1,2,3.
	$$
\end{prop}

\begin{proof}
	Assume that $C$ has equation
	$$F(x, y)=a(1+x^3y^3)+b(xy+x^2y^2)+c(x^2+ xy^3)+d(y^2+x^3y)=0.$$
		Setting $x=0$ and $y=0$ respectively in the equation of $C$, we obtain 6 points denoted by $P_1, P_2, P_3$ (for $x=0$), and $Q_1, Q_2, Q_3$ (for $y=0$) respectively. The point $\xi_C\in \Jac (C)$ is the class of 
		$$D:=P_1+P_2+P_3-Q_1-Q_2-Q_3.$$
	
	By the explicit constructions in \ref{A_2}, \ref{lem:K to A1A1}, together with the decomposition in Proposition \ref{prop:jacobian-decomposition-N}, we have explicit Weierstrass forms for $E_0$ and $E_1$ as follows:
	
	\begin{lem}
$$
		E_0:	Y^{2} =X^{3}+p_0 X+q_0 ,\qquad E_1:	Y^{2}  =X^{3}+p_1X+q_1,$$
		where the coefficients are related to coefficients of $C$ by the following formulas:
		$$
		p_0 =A_0 C_0-B_0^{2} / 3 \qquad 	q_0=\frac{2 B_0^{3}-9 A_0 B_0 C_0+27 A_0^{2} D_0}{27} ,$$
		$$A_0=4(c-d)^{2}\neq0,\qquad B_0=-4 c d+(3 a-b)^{2},\qquad C_0=2 a(3 a-b),\qquad D_0=a^2.$$
$$
		p_1 =A_1C_1-B_1^{2} / 3 \qquad q_1=\frac{2 B_1^{3}-9 A_1 B_1C_1+27 A_1^{2} D_1}{27} ,$$
		$$A_1=4(a+b)^2-4(c+d)^{2}\neq0,\quad B_1=-4 c d+(a+b)(9a+b),\quad C_1=6a^2+2ab,\quad D_1=a^2.
	$$
	\end{lem}
	\begin{proof}
	 Assume that $C$ has equation
	 $$F(x, y)=a(1+x^3y^3)+b(xy+x^2y^2)+c(x^2+ xy^3)+d(y^2+x^3y)=0.$$
	 By \ref{A_2} the genus 2 curve $X_1=C/\sigma_1$ has equation 
	 \begin{align*}w^2=&\left(a\left(1+u^{3}\right)+b u(1+u)\right)^{2}-4(c+d u)(c u+d) u^{2}\\
	 	=&a^2(u^6+1) + 2ab(u^5+u) + (b^2+2ab-4dc)(u^4+u^2)
	 	+ (2a^2+2b^2-4c^2-4d^2)u^3
	 \end{align*}
 Then we apply the substitution
 $$ x_0:= \frac{-u}{(u+1)^2} =\frac{xy}{(xy+1)^2},\qquad y_0=\frac{w}{(u+1)^3}=\frac{x^2(c+dxy)-y^2(cxy+d)}{(1+xy)^3}$$
 Then we have equation 
 \begin{align*}
 y_0^2=&((3a-b)x_0+a)^2+4((c-d)^2x_0^3+cdx_0^2)\\
 =&A _0x_0^3+B_0x_0^2+C_0x_0+D_0
 \end{align*}
 where
 $$A_0=4(c-d)^{2}\neq0,\qquad B_0=-4 c d+(3 a-b)^{2},\qquad C_0=2 a(3 a-b),\qquad D_0=a^2.$$
Let
$$X=A_0x_0+B_0/3, \qquad Y=A_0y_0,$$  we obtain a Weierstrass equation for $E_0$ as in the lemma. 

For $E_1$, similarly we will apply the substitution
 $$ x_1:= \frac{u}{(u-1)^2} =\frac{xy}{(xy-1)^2},\qquad y_1=\frac{w}{(u-1)^3}=\frac{x^2(c+dxy)-y^2(cxy+d)}{(xy-1)^3}$$
 to obtain an equation for $E_1$:
 $$y_1^2=A_1x_1^3+B_1x_1^2+C_1x_1+D'$$
 with 
 $$A_1=4(a+b)^2-4(c+d)^{2}\neq0,\quad B_1=-4 c d+(a+b)(9a+b),\quad C_1=6a^2+2ab,\quad D_1=a^2.
 $$
 And let  $X=A_1x_1+B_1/3$, $Y=A_1y_1$.
 We obtain a Weierstrass equation for $E_1$ as in the lemma.
 	 \end{proof}

	The $P_i$ has image $P_{E_0}:(B_0/3, A_0a)$ on $E_0$, while $Q_i$ has image $Q_{E_0}:(B_0/3, -Aa)=-P_{E_0}$ on $E_0$. Then $p_0(\xi)=6P_{E_0}$ on $E_0$. A direct calculation shows that, the $X$-coordinate of $2P_{E_0}$ is 
		
		$$B_0/3-4cd.$$
		
		We know that when $cd=0$, the divisor $D=P_1+P_2+P_3-Q_1-Q_2-Q_3$ has image $D_{E_0}=P_{E_0}-Q_{E_0}=2P_{E_0}$, which is a torsion point on $E_0$. So the divisor $D_{E_0}$ on $C$ is torsion. 	But when $cd\neq 0$, it may happen that $D_{E_0}$ is not torsion. 
		
		We proceed similarly on $E_1$. Then $p_1(\xi)=6P_{E_1}$ while $P_{E_1}$ has coordinate $(B_1/3, A_1a)$. A direct calculation shows that, the $X$-coordinate of $2P_{E_1}$ is 
		
		$$B_1/3+4cd-4ab.$$
		But when $ab-cd\neq 0$, it may happen  that $D_{E_1}$ is non-torsion. We will complete the proof  by giving two examples
\end{proof}
\begin{example}
	Let $(a, b, c, d)=(1, 1,1, 2)$, then we have 
	$$(A_0, B_0, C_0, D_0)=(4, -4, 4, 1), \qquad (p_0, q_0)=(32/3, 880/27),$$
	 $$ (A_1, B_1, C_1, D_1)=(-20, 12, 8, 1), \qquad (p_1, q_1)=(-208, 1168).$$
	The elliptic curve $E_0$ has an equation
	$$y^2=x^3+32/3x+880/27$$
	with $j$-invariant $j=250.137404580153$.
	The elliptic curve $E_1$ has an equation
	$$y^2=x^3-208x+1168$$
	with $j$-invariant $j = -242970624/3275$.
	
	Comparing the two $j$-invariants of $E_0$ and $E_1$, they are not equal.
	The point $P_{E_0}$ has coordinate $(-4/3, 4)$. Using Sage, this point has a N\'{e}ron-Tate height of 0.216095198667748, thus not torsion. Hence $D_{E_0}$ is non-torsion  on $E_0$.
	
	The point $P_{E_1}$ has coordinate $(4, 20)$. Using Sage, this point has a N\'{e}ron-Tate height of 0.106438087886740, thus not torsion. Hence  $D_{E_1}$ is non-torsion  on $E_1$.
\end{example}
\begin{example}
	Let $(a, b, c, d)=(3, 1, 2, 5)$, then we have
	$$(A_0, B_0, C_0, D_0) = (3, 24, 48, 9), \qquad (p_0, q_0)=(1536,  -1136),
	$$
	$$(A_1, B_1, C_1, D_1) = (-132, 72, 60, 9), \qquad (p_1, q_1)=(-9648, 374544).$$
	
	The elliptic curve $E_0$ has an equation
	$$y^2=x^3+1536x+-1136$$
	with $j$-invariant $j=134217728/77859$.
	The elliptic curve $E_1$ has an equation
	$$y^2=x^3-9648x+374544$$
	with $j$-invariant $j = -33261981696/1046771$.
	Comparing the two $j$-invariants of $E_0$ and $E_1$, they are not equal. Comparing with the example above, $j$-invariants for the two $E_0's$ (resp. $E_1's$) are not equal.
	
	The point $P_{E_0}$ has coordinate $(8, 108)$. Using Sage, this point has a N\'{e}ron-Tate height of 0.727274930661652, thus not torsion. Hence $D_{E_0}$ is non-torsion on $E_0$.
	
	The point $P_{E_1}$ has coordinate $(24, 396)$. Using Sage, this point has a N\'{e}ron-Tate height of 0.641675355328461, thus not torsion. Hence $D_{E_1}$ is non-torsion  on $E_1$.
\end{example}

\begin{proposition}
	The $\xi_C$ is non-torsion for generic $C \in \sN$.
\end{proposition}

\begin{proof}
Composing the isogeny $\Jac(C)\to E_i$ with the map $C\to \Jac(C)$, the image of $\xi$ in $E_i$ is also a degree zero divisor $\xi_E$. In the group $E(K)$, if $\xi_E$ is non-torsion then $\xi_C\in \Jac(C)(\bar{\BQ})$ is also non-torsion. In fact, if $\varphi: G\to H$ is a homomorphism of two finitely generated groups, and $x\in G$ is torsion, then $\varphi(x)$ is also torsion. Conversely, if a point $y\in H$ is non-torsion, then any preimage of $y$ in $G$ is non-torsion. 
\end{proof}

Geometrically, Proposition~\ref{prop:nonzero-projections} shows that the canonical quadratic point sees all four elliptic directions in the Jacobian decomposition. This is precisely the input needed to identify the minimal abelian subscheme attached to the section $\nu$.

\begin{corollary}
	\label{cor:minimal-J}
	For the family $\pi:\sJ \to \sN$, the minimal abelian subscheme $A_{\min}\incl \sJ$ such that $\nu(\sN)\subset A_{\min}+\tau$ with $\tau$ a 
	torsion section of $\pi$  is $\prod_{i=0}^3\CE_i$.
\end{corollary}

\begin{proof}
	If $A_{\min}$ were a proper abelian subscheme, then on a general fiber, $A_{\min, C}$ would be a proper abelian subscheme of $\Jac(C)$. Under the decomposition of Proposition~\ref{prop:jacobian-decomposition-N}, such a proper subscheme would miss at least one elliptic direction up to isogeny. But Proposition~\ref{prop:nonzero-projections} shows that $\xi_C$ has nonzero projection to every factor, contradicting the definition of the minimal abelian subscheme.
\end{proof}

\subsection{Proof of Theorem \ref{thm:intro-triple}}

Now we verify the condition in Theorem \ref{thm: criterion}.
Let $B \subseteq A_{\min}$ be an abelian subscheme over $\sN$. 
The modular maps
\[
\varphi_B : \sN \longrightarrow \sA_{q_B}, \qquad \phi_B: \sN\lra \CS_{q_B}
\]
attached to the quotient $\sJ_\sN/B$ record the variation of the corresponding product of elliptic factors. 
Because the family $\sN$ has dimension $3$,  the condition of the criterion is that for every proper quotient $\sJ/B$, \[
\dim\sN\le \dim \varphi_B(\sN)+\dim_\sN B.
\]

Corollary \ref{cor:minimal-J}  implies that the minimal abelian subscheme $A_{\min}\incl \sJ_\sN$  for $\nu$ is $\sJ$ itself.
Also, since $\sJ_\sN$ is isogenous to the product  $\prod_{i=0}^{3} \CE_i$ with $\CE_i$ generically non-constant and non-isogenous to each other, each $B$ must be isogenous to a partial product $\prod _{i\in I} \CE_i$.  Thus $\varphi_B$ and $\phi_B$ factor through  the following maps:
$$\varphi_I: \sN\lra \CE^4\lra \CE^{4-|I|}, \qquad J_I: \sN \lra \BA^4\lra \BA^{4-|I|}$$ where $\CE$ is the universal elliptic curve over $\BA^1$ after adding some level structure, and $\phi_I$ is the map by taking $j$-invariants for factors not in $I$. It follows that 
$$\dim \varphi_B (\sN)=\dim \varphi_I(\sN)\ge \dim J_I(\sN).$$
Thus, we are reducing to show the following
$$\dim J_I(\sN)\ge 3-|I|.$$
This follows from Proposition \ref{prop:jacobian-decomposition-N}, because $\dim J(\sN)=3$, and the projection $\BA^4\lra \BA^{4-|I|}$ can at most drop the dimension $J(\sN)$ by the $|I|$.

\section{A CM family}\label{sec:cm}
This section is devoted to a second application of Theorem~\ref{thm: criterion}, namely to families containing a special CM curve with $\zeta_{15}$-action.
\subsection{On $y^5=x^3-1$}
We begin with the Jacobian of the curve $C_0$  defined by 
$$y^5=x^3-1.$$
Let $\zeta_n$ denote a primitive $n$-th root of unity.	Then there are two automorphisms: 
$$(x, y)\mapsto (x, \zeta_5 y);$$
$$(x, y)\mapsto (\zeta_3x, y).$$
Thus the Jacobian $J(C_0)$ has  $\BQ[\zeta_5]\subset \End(J(C_0))$ which gives a  half-CM structure and $\BQ[\zeta_3]\subset \End(J(C_0))$.
The basis of $H^0(C_0, \Omega^1)$ is given by 
$$ \omega_1= \frac{1}{x} dy,\qquad  \omega_2=\frac{1}{x^2} dy,\qquad   \omega_3=\frac{y^2}{x^2} dy,\qquad  \omega_4=\frac{y}{x^2} dy.$$
Thus, the canonical embedding of $C_0$ lives in the cone 
$$\omega_2\cdot \omega_3=\omega _4^2, $$
and the induced action of   $\BQ[\zeta_{15}]$ on the Jacobian of $C_0$ is given as follows: 
\begin{equation}\label{eq-cm}(\omega_1, \omega_2,\omega_3, \omega_4)\mapsto (\zeta_{15}^{-2}\omega_1, \zeta_{15}^{-7}\omega_2,\zeta_{15}^{-1}\omega_3, \zeta_{15}^{-4}\omega_4).\end{equation}
So there is a degree-8 semisimple subalgebra $\BQ[\zeta_{15}]\subset \End_\BQ(J(C_0))$ acting on $J(C_0)$, so $J(C_0)$ is CM.

\begin{lem}\label{lem-cm}
The Jacobian $\Jac (C_0)$ is simple. 
\end{lem}
\begin{proof}
We prove this by considering the embedding 
	$$\BQ(\zeta_{15})\to \End(\Jac(C))\otimes\BQ$$
	induced by the action of $\zeta_{15}$ on $C$. If $\Jac(C)$ is not simple, then it is isogenous to a direct sum $\oplus_{i=1}^n A_i^{r_i}$ where $A_i$ are simple abelian varieties not isogenous to each other. In this way, 
	$$\End(\Jac(C_0))\otimes \BQ\cong \bigoplus_{i=1}^n M_{r_i} (D_i)$$
	where $D_i$ are the division algebras $\End(A_i)\otimes \BQ$. The maximal subalgebra of $\End(\Jac(C))\otimes \BQ$ has the form 
	$L=\bigoplus_{i=1}^n L_i$, with $L_i$ a maximal commutative subalgebra of $M_{r_i}(D_i)$. 
	As $\dim_\BQ L_i\leq 2r_i\dim A_i$, we have  $\dim_\BQ L\leq \sum 2r_i \dim A_i=2\dim \Jac(C)=8$.
	
	As $\dim_\BQ\BQ(\zeta_{15})=8$, we see that the image of $\BQ(\zeta_{15})$ is a maximal subalgebra of $\End(\Jac(C))\otimes \BQ$. So we have the form $\BQ(\zeta_{15})=\bigoplus_{i=1}^nL_i$. As $\BQ(\zeta_{15})$ is a field, we have $n=1$. And $A_1$ is a CM abelian variety by a proper CM subfield $M$ of $\BQ(\zeta_{15})$ with $r_1=[\BQ(\zeta_{15}):M]$. Moreover, the CM-type on $\Jac(C_0)$ is stable under the action of $\Gal(\BQ(\zeta_{15})/M)$.
	
	By computation in \ref{eq-cm}, we know that the CM type is given by 
	$$(b_1, b_2, b_3, b_4)=(-1, -2, -4, -7).$$
	The stabilizer subgroup $H$ of $\Gal(\BQ(\zeta_{15})/\BQ)$ of the CM type is given by $r\in (\BZ/15\BZ)^\times$ such that 
	$$\{rb_1, rb_2, rb_3, rb_4\}=\{b_1, b_2, b_3, b_4\}.$$
	So $r$ can only be 1, 2, 4, 7.
	We calculate the left-hand sides  as follows:
	$$r=2: (-2, -4, 7, 1),$$
	$$r=4: (-4, 7, -1, 2),$$
	$$r=7: (-7, 1, 2, -4).$$
	So, none of them matches the right-hand side.
	Thus, the stabilizer $H$ is trivial. It follows that $r_1=1$.
	Thus $\Jac (C_0)$ is simple.

\end{proof}

\begin{lem}\label{lem-cm}
	The $C_0$ is the only genus four curve over $\BC$ with  simple Jacobian and action by $\mu_{15}$.\end{lem}
\begin{proof}
Let $C$ be a smooth curve of genus $4$ with simple Jacobian and action by $\mu_{15}$. Then we have quotient curves
$$\xymatrix{
		C\ar[d]\ar[r]&C/\zeta _5\ar[d]\\
		C/\zeta _3\ar[r]&C/\zeta _{15}
	}$$
The Jacobians of these curves are proper subvarieties of $\Jac (C)$. As $\Jac (C)$ is simple, these curves must have genus $0$.
Thus, we have an isomorphism $C/\zeta _5\iso\BP^1$. On the other hand, $C/\zeta_5$ has an action by $\zeta _3$. This will induce an action of $\zeta _3$ on $\BP^1$. Composing with an automorphism of $\BP^1$, we may assume that $\zeta _3$ acting on $\BP^1$ is given by multiplication on the coordinate $x$. 
Using the Hurwitz formula for the projection $C\to C/\zeta _5=\BP^1$, we have 
	$$2\times 4-2=5\times (-2)+R,\qquad R=16.$$
	As each ramification point on $\BP^1$ has ramification index $4$, we have four ramification points on $\BP^1$.
	This set is stable under the action of $\zeta _3$. Thus, one of them is $0$ or $\infty$, and the others are in one orbit under $\zeta _3$.
	Composing with another automorphism of $\BP^1$, we may assume this set is 
	$$1, \zeta _3, \zeta _3^2, \infty.$$
By Kummer's theory, $C$ must have a function $y^5=f(x)$, for $f\in \BC[x]$ is a monic polynomial without a $5$th power factor. 
The zeros of $f$ are exactly $1, \zeta _3, \zeta _3^2$.
The action by $\zeta _3$ implies that $f(\zeta _3 x)=f(x)$. Thus $f$ is a power of  $(x^3-1)$. Thus $y^5=(x^3-1)^k$ for a $k$ prime to $5$.
Choose integer $a, b$ so that $ak+5b=1$, then change $y$ to $(x^3-1)^b y^a$, we obtain 
$$y^5=x^3-1.$$
\end{proof}

\subsection{ Proof of  theorem \ref{thm:intro-cm}}

Let $C_0$ be the CM curve $y^5=x^3-1$ by $\zeta_{15}$. And let $C_1$ be the explicit curve in $\sN$ constructed before in the chapter \ref{sec:triple} such that $\xi_{C_1}$ is non-torsion.

In this section, we aim to prove the following theorem:

\begin{theorem}
	Let $S$ be any family of smooth curves of genus 4 containing $C_0$, $C_1$ with $\dim S\le 4$, then $\xi$ is big on $S$.
\end{theorem}

\begin{proof}
	We want to apply Theorem \ref{thm: criterion}. 
	Recall that $A_{\min}$ is the minimal abelian subscheme of $\sJ$ over $S$ so that $\nu(S)$ is contained in a translation $\tau+A_{\min, S}$ by a torsion section $\tau$ of $\sJ_S$ over $S$.
	By Lemma \ref{lem-cm}, $\Jac (C_0)$ is simple. Thus,  in the criterion of Theorem \ref{thm: criterion}, we must have  $A_{\min}=\sJ_S$, which  is simple.
	Let $B$ be an abelian subscheme of $A_{\min}$; then either $B=\sJ_S$ or $B=0$. 
	\begin{itemize}
		\item When $B=\sJ_S$, $\dim\varphi_B(S)=0$ and $\dim_SB=\dim_S\sJ_S=4$. Of course $\dim S\leq \dim_S\sJ_S$. 
		\item When $B=0$, $\dim_SB=0$, then  $\varphi_B$ is injective. Thus $\dim S=\dim \varphi_B( S)$.
	\end{itemize}
	Thus, the criterion \ref{thm: criterion}  implies that $\xi$ is big. 
\end{proof}

		\section{Consequences and further questions}\label{sec:consequences}
		
		In this final section, we record several arithmetic consequences of the bigness results proved in Sections~\ref{sec:triple} and~\ref{sec:cm}, and then briefly discuss some questions for future study. The general philosophy is that once the canonical quadratic point is big on a family, its Néron--Tate height cannot remain too small compared with the arithmetic complexity of the corresponding curve. This leads naturally to finiteness statements and non-torsion results.
		
		\subsection{A Northcott-type consequence}
		
		We begin with a general consequence of bigness.
		
		\begin{proposition}
			\label{prop:northcott-general}
			Let $S$ be an irreducible family of smooth non-hyperelliptic genus $4$ curves, and assume that the canonical quadratic point $\xi$ is big on $S$. Then there exist a Zariski open dense subset $U \subseteq S$, a real number $\epsilon>0$, and a constant $c \in \RR$ such that for all $C \in U(\overline{K})$,
			\[
			\hat h(\xi_{C}) \ge \epsilon\, h_{\mathrm{Fal}}(C)-c.
			\]
			In particular, for every real number $B$, the set
			\[
			\left\{\, s \in U(\overline{K}) : h_{\mathrm{Fal}}(C)\le B \text{ and } [K(s):K]\le d \,\right\}
			\]
			is finite for each fixed positive integer $d$.
		\end{proposition}
		\begin{proof}
			The lower bound is precisely the definition of bigness. Once $h_{\mathrm{Fal}}(C)$ is bounded above, the inequality shows that $\hat h(\xi_{C})$ is also bounded above on the same set of points. Since bounded degree and bounded Faltings height give a Northcott-type finiteness statement for the corresponding moduli points, the desired finiteness follows.
		\end{proof}
		
		Applying Proposition~\ref{prop:northcott-general} to the families studied earlier, we obtain the following.
		
		\begin{corollary}
			\label{cor:northcott-N}
			Let $N$ be the triple involution locus of Section~\ref{sec:triple}. Then there exists a Zariski open dense subset $U \subseteq N$ such that for every real number $B$ and every positive integer $d$, the set
			\[
			\left\{\, s \in U(\overline{K}) : h_{\mathrm{Fal}}(C)\le B \text{ and } [K(s):K]\le d \,\right\}
			\]
			is finite.
		\end{corollary}
	\begin{proof}
		This follows immediately from Theorem~\ref{thm:intro-triple} and Proposition~\ref{prop:northcott-general}.
	\end{proof}
	
	\begin{corollary}
		\label{cor:northcott-CM}
		Let $S$ be a family as in Theorem~\ref{thm:intro-cm}. Then there exists a Zariski open dense subset $U \subseteq S$ such that for every real number $B$ and every positive integer $d$, the set
		\[
		\left\{\, s \in U(\overline{K}) : h_{\mathrm{Fal}}(C)\le B \text{ and } [K(s):K]\le d \,\right\}
		\]
		is finite.
	\end{corollary}
	
	\begin{proof}
		This is immediate from Theorem~\ref{thm:intro-cm} and Proposition~\ref{prop:northcott-general}.
	\end{proof}
\subsection{Transcendental points and non-torsion}

We next record a consequence for complex fibers corresponding to transcendental points of the parameter space.

\begin{proposition}
	\label{prop:transcendental}
	Let $S$ be an irreducible family of smooth non-hyperelliptic genus $4$ curves over a number field, and assume that $\xi$ is big on an open subset  $U\subset S$. 
 If $s \in U(\CC)$ is not algebraic over the ground field, then the canonical quadratic point $\xi_{C}$ is non-torsion.
\end{proposition}

\begin{proof}
	Assume for contradiction that $\xi_{C}$ is torsion for some transcendental point $s \in U(\BC)$. Let $X$ be the Zariski closure of $s$.
	Then the Néron--Tate height of $\xi_x$ is zero for every $x\in X(\bar \BQ)$. This contradicts to the Northcott property.
	\end{proof}

	\end{document}